\documentclass[12pt]{amsart}
\usepackage{amsmath,amssymb,amsbsy,amsfonts,latexsym,amsopn,amstext,
                                               amsxtra,euscript,amscd}

\newfont{\teneufm}{eufm10}
\newfont{\seveneufm}{eufm7}
\newfont{\fiveeufm}{eufm5}
\newfam\eufmfam
                \textfont\eufmfam=\teneufm \scriptfont\eufmfam=\seveneufm
                \scriptscriptfont\eufmfam=\fiveeufm
\def\bbbc{{\mathchoice {\setbox0=\hbox{$\displaystyle\rm C$}\hbox{\hbox
to0pt{\kern0.4\wd0\vrule height0.9\ht0\hss}\box0}}
{\setbox0=\hbox{$\textstyle\rm C$}\hbox{\hbox
to0pt{\kern0.4\wd0\vrule height0.9\ht0\hss}\box0}}
{\setbox0=\hbox{$\scriptstyle\rm C$}\hbox{\hbox
to0pt{\kern0.4\wd0\vrule height0.9\ht0\hss}\box0}}
{\setbox0=\hbox{$\scriptscriptstyle\rm C$}\hbox{\hbox
to0pt{\kern0.4\wd0\vrule height0.9\ht0\hss}\box0}}}}
\def\bbbq{{\mathchoice {\setbox0=\hbox{$\displaystyle\rm
Q$}\hbox{\raise 0.15\ht0\hbox to0pt{\kern0.4\wd0\vrule
height0.8\ht0\hss}\box0}} {\setbox0=\hbox{$\textstyle\rm
Q$}\hbox{\raise 0.15\ht0\hbox to0pt{\kern0.4\wd0\vrule
height0.8\ht0\hss}\box0}} {\setbox0=\hbox{$\scriptstyle\rm
Q$}\hbox{\raise 0.15\ht0\hbox to0pt{\kern0.4\wd0\vrule
height0.7\ht0\hss}\box0}} {\setbox0=\hbox{$\scriptscriptstyle\rm
Q$}\hbox{\raise 0.15\ht0\hbox to0pt{\kern0.4\wd0\vrule
height0.7\ht0\hss}\box0}}}}
\def\bbbt{{\mathchoice {\setbox0=\hbox{$\displaystyle\rm
T$}\hbox{\hbox to0pt{\kern0.3\wd0\vrule height0.9\ht0\hss}\box0}}
{\setbox0=\hbox{$\textstyle\rm T$}\hbox{\hbox
to0pt{\kern0.3\wd0\vrule height0.9\ht0\hss}\box0}}
{\setbox0=\hbox{$\scriptstyle\rm T$}\hbox{\hbox
to0pt{\kern0.3\wd0\vrule height0.9\ht0\hss}\box0}}
{\setbox0=\hbox{$\scriptscriptstyle\rm T$}\hbox{\hbox
to0pt{\kern0.3\wd0\vrule height0.9\ht0\hss}\box0}}}}
\def\bbbs{{\mathchoice
{\setbox0=\hbox{$\displaystyle     \rm S$}\hbox{\raise0.5\ht0\hbox
to0pt{\kern0.35\wd0\vrule height0.45\ht0\hss}\hbox
to0pt{\kern0.55\wd0\vrule height0.5\ht0\hss}\box0}}
{\setbox0=\hbox{$\textstyle        \rm S$}\hbox{\raise0.5\ht0\hbox
to0pt{\kern0.35\wd0\vrule height0.45\ht0\hss}\hbox
to0pt{\kern0.55\wd0\vrule height0.5\ht0\hss}\box0}}
{\setbox0=\hbox{$\scriptstyle      \rm S$}\hbox{\raise0.5\ht0\hbox
to0pt{\kern0.35\wd0\vrule height0.45\ht0\hss}\raise0.05\ht0\hbox
to0pt{\kern0.5\wd0\vrule height0.45\ht0\hss}\box0}}
{\setbox0=\hbox{$\scriptscriptstyle\rm S$}\hbox{\raise0.5\ht0\hbox
to0pt{\kern0.4\wd0\vrule height0.45\ht0\hss}\raise0.05\ht0\hbox
to0pt{\kern0.55\wd0\vrule height0.45\ht0\hss}\box0}}}}
\def\bbbz{{\mathchoice {\hbox{$\sf\textstyle Z\kern-0.4em Z$}}
{\hbox{$\sf\textstyle Z\kern-0.4em Z$}} {\hbox{$\sf\scriptstyle
Z\kern-0.3em Z$}} {\hbox{$\sf\scriptscriptstyle Z\kern-0.2em
Z$}}}}

\newtheorem{theorem}{Theorem}
\newtheorem{lemma}[theorem]{Lemma}

\def\squareforqed{\hbox{\rlap{$\sqcap$}$\sqcup$}}
\def\qed{\ifmmode\squareforqed\else{\unskip\nobreak\hfil
\penalty50\hskip1em\null\nobreak\hfil\squareforqed
\parfillskip=0pt\finalhyphendemerits=0\endgraf}\fi}

\def\cE{{\mathcal E}}

\def\cI{{\mathcal I}}

\def\cP{{\mathcal P}}

\def\cX{{\mathcal X}}

\def \sf {\mathfrak s}
\def \fG {\mathfrak G}

\def\ssumPI{\mathop
{\sum_{g_1,\ldots, g_{k-2m} \in \cP_d}\, 
\sum_{h_1, \ldots, h_{2m} \in \cI_d}}}

\newcommand{\ignore}[1]{}

\def \F{\mathbb{F}}

\def\\{\cr}
\def\({\left(}
\def\){\right)}

\def\rf#1{\left\lceil#1\right\rceil}

\begin{document}


\title[Cayley Graphs Generated by Small Degree Polynomials]
{Cayley Graphs Generated by Small Degree Polynomials over 
Finite Fields}
 
 \author[I. E. Shparlinski] {Igor E. Shparlinski}

\address{Department of Pure Mathematics, University of New South Wales,
Sydney, NSW 2052, Australia}
\email{igor.shparlinski@unsw.edu.au}

\begin{abstract} We improve upper bounds  of F.~R.~K.~Chung and
of  M.~Lu, D.~Wan, L.-P.~Wang, 
X.-D.~Zhang on the diameter of some Cayley graphs constructed 
from polynomials over finite fields.
\end{abstract}


\maketitle

\section{Introduction}

Let $\cP_d$ be the set of monic polynomials  of degree $d$ 
over a finite field $\F_q$ of $q$ elements, that are powers
of some irreducible polynomial, that is 
\begin{equation*}
\begin{split}
\cP_d  = \{g\in \F_q[X]~:~ &\deg g = d, \ g = h^k, \\
  h & \in \F_q[X]\
\text{monic and irreducible,}\ k    =1,2, \ldots,\ \}.
\end{split}
\end{equation*}

For a root $\alpha$ of 
an irreducible polynomial $f \in \F_q[X]$ of degree $n$,  
thus $\F_q(\alpha) = \F_{q^n}$, 
we define  
$$\cE(\alpha,d)  = \{g(\alpha)~:~g \in \cP_d\}.
$$ 
It is easy to see that for $d < n$ we have
$$
\# \cE(\alpha,d)  = \#\cP_d  = (1 + o(1)) \frac{q^d}{d}
$$
as $d \to \infty$, see also~\eqref{eq:Id} below.

Following Lu, Wan, Wang and Zhang~\cite{LWWZ}, we now define the directed 
Cayley graph  $\fG(\alpha,d)$  on $q^n-1$ vertices,
labelled by the elements of $\F_{q^n}^*$, where for $u,v \in \F_{q^n}^*$ the edge $u\to v$
exists if and only if $u/v \in \cE(\alpha,d)$. These graphs are similar  to 
those introduced by Chung~\cite{Chung} however are a little spraser: 
they are  $\#\cP_d$-regular rather than  $q^d$-regular as in~\cite{Chung}.

It has been shown in~\cite{LWWZ} that the graphs $\fG(\alpha,d)$ have very attractive connectivity
properties. In particular, we denote by $D(\alpha,d)$ the {\it diameter\/} of 
$\fG(\alpha,d)$.
Using bounds of multiplicative character sum from~\cite[Theorem~2.1]{Wan},
Lu, Wan, Wang and Zhang~\cite{LWWZ} have shown that for $n < q^{d/2} + 1$
the graph $\fG(\alpha,d)$ is connected and its diameter satisfies the inequality
\begin{equation}
\label{eq:Diam}
D(\alpha,d) \le \frac{2n}{d}\(1 + \frac{2\log(n-1)}{d \log q - 2\log(n-1)}\) + 1. 
\end{equation}

Here we augment the argument of~\cite{LWWZ} with some new combinatorial 
and analytic 
considerations and improve the bound~\eqref{eq:Diam}. 

First we  assume that $d \ge 2$. 

\begin{theorem}\label{thm:Diam d}
For $d \ge 2$ and a root $\alpha$ of an  irreducible polynomial 
$f \in \F_q[X]$ of degree $\deg f = n$ with
$2d+1 \le n <   q^{d/2} + 1$, we have
$$
D(\alpha,d) \le \frac{2n}{d}\(1 + \frac{\log(n-1)-1}{d \log q - 2\log(n-1)}\)
+ \frac{4\log(n-1) + 7}{d \log q - 2\log(n-1)}. 
$$ 
\end{theorem}

For $d=1$ the bound~\eqref{eq:Diam} is   exactly the same as the 
bound of Wan~\cite[Theorem~3.3]{Wan} which improves slightly the
bound of Chung~\cite[Theorem~6]{Chung}.  
For $d=1$,  we  set $\Delta(\alpha)= D(\alpha,1)$.
For a sufficiently  large $q$, 
Katz~\cite[Theorem~1]{Katz} has improved the results of Chung~\cite{Chung}
and showed that $\Delta(\alpha) \le n+2$, provided that $q\ge B(n)$ for some 
inexplicit function $B(n)$ of $n$. 
Furthermore, Cohen~\cite{Coh}   shows that one can take $B(n) = (n(n+2)!)^2$
in the estimate of Katz~\cite{Katz}. 

We also use our idea in the case $d=1$ and obtain an improvement of~\eqref{eq:Diam} 
and thus of the bounds of  Chung~\cite[Theorem~6]{Chung} and  Wan~\cite[Theorem~3.3]{Wan}. 

\begin{theorem}\label{thm:Diam 1}
For  a root $\alpha$ of an  irreducible polynomial 
$f \in \F_q[X]$ of degree $\deg f = n$ with
$3\le  n <   q^{1/2} + 1$, we have
$$
\Delta(\alpha) \le 2n \(1 + \frac{\log(n-1)-1}{\log q - 2\log(n-1)}\)
+ \frac{3\log(n-1) + 3}{ \log q - 2\log(n-1)}.
$$ 
\end{theorem}

We use the same idea for the proofs of Theorems~\ref{thm:Diam d}
and~\ref{thm:Diam 1}, however the technical details are slightly different.

We also note that the additive constants $7$  and $3$ 
in the bounds of Theorems~\ref{thm:Diam d} and~\ref{thm:Diam 1},
respectively, 
can be replaced by a slightly smaller 
(but fractional values). 

To compre the bound~\eqref{eq:Diam}
with Theorems~\ref{thm:Diam d} and~\ref{thm:Diam 1},
we assume that $n = q^{(\vartheta+o(1))d}$ for some fixed
positive $\vartheta < 1/2$.

The  Theorems~\ref{thm:Diam d} and~\ref{thm:Diam 1}, imply that 
for any $d\ge 1$, 
$$
D(\alpha,d) \le \(\frac{2-2\vartheta}{1-2\vartheta}+o(1)\)\frac{n}{d}, 
$$
while~\eqref{eq:Diam} implies a weaker bound 
$$
D(\alpha,d) \le \(\frac{2}{1-2\vartheta}+o(1)\)\frac{n}{d}.
$$

\section{Preparation}

We define the polynomial analogue of the von Mangoldt function
as follows. For $g \in \F_q[X]$ we define
$$
\Lambda(g)=
\begin{cases}
\deg h,   & \quad\text{if}~g = h^k\ \text{for some irreducible}\ h \in \F_q[X], \\
       0,   & \quad\text{otherwise}. 
\end{cases}
$$

Let $\cX_n$ be the set of multiplicative characters of $\F_{q^n}$
and let $\cX_n^* = \cX_n \setminus \{\chi_0\}$ be the 
set of non-principal characters; 
we appeal to~\cite{IwKow} for a background 
on the basic properties of multiplicative characters, such as orthogonality. 

For any $\chi \in \cX_n$ we also define the character sum
$$
S_{\alpha,d}(\chi) = \sum_{g \in \cP_d} \Lambda(g) \chi\(g(\alpha)\).
$$
A simple combinatorial argument shows that for the principal 
character $\chi_0$ we have
\begin{equation}
\label{eq:sum Lambda}
S_{\alpha,d}(\chi_0) = \sum_{g \in \cP_d} \Lambda(g)  = q^d,
\end{equation}
see, for example,~\cite[Corollary~3.21]{LN}.

As in~\cite{LWWZ}, we recall that by~\cite[Theorem~2.1]{Wan} we have:

\begin{lemma}\label{lem:Weil}
For any $\chi \in \cX_n^*$ we have
$$
\left|S_{\alpha,d}(\chi) \right| \le (n-1) q^{d/2}. 
$$
\end{lemma}

We also consider the set $\cI_d$ of irreducible polynomials
of degree $d$, that is, 
$$
\cI_d  = \{h\in \F_q[X]~:~\deg h = d, \   h \in \F_q[X]\
\text{irreducible}\},
$$
and the sums
$$
T_{\alpha,d}(\chi) = \sum_{h \in \cI_d} \chi\(h(\alpha)\).
$$
Our new ingredient is the following bound 
``on average''. 

\begin{lemma}\label{lem:Moment}
Let  $m = \rf{n/d}-1$. 
Then 
$$
\sum_{\chi \in \cX_n} \left|T_{\alpha,d}(\chi) \right|^{2m} \le m! (q^n-1) (\# \cI_d)^m.
$$
\end{lemma}

\begin{proof} Using the orthogonality of characters, we see that 
$$
\sum_{\chi \in \cX_n} \left|T_{\alpha,d}(\chi) \right|^{2m} =
(q^n-1)N, 
$$
where $N$ is the number of solutions to the equation
$$
h_1(\alpha)\ldots h_m(\alpha) = h_{m+1}(\alpha)\ldots h_{2m}(\alpha),
$$
with some $h_1,\ldots, h_{2m} \in \cI_d$.
Since $dm < n$ this implies the identity 
$$
h_1(X)\ldots h_m(X) = h_{m+1}(X)\ldots h_{2m}(X)
$$
in the ring of polynomials over $\F_q$. 
Thus, using the uniqueness of polynomial factorisation, 
we obtain 
$$
W \le m! (\# \cI_d)^m, 
$$
which concludes the proof.
\end{proof} 

Finally, we recall the well-know formula (see, for example,~\cite[Theorem~3.25]{LN}) 
\begin{equation}
\label{eq:Id}
\# \cI_d = \frac{1}{d}\sum_{s\mid d} \mu(s) q^{d/s}, 
\end{equation}
where  $\mu(s)$ is  the M{\" o}bius function, that is,
$$
\mu(s)= \begin{cases} (-1)^{\nu} & \text{if } s \
\text{is a product $\nu$ distinct primes}, \\
0 & \text{otherwise}.
\end{cases}
$$

\section{Proof of Theorem~\ref{thm:Diam d}}

Let as before  $m = \rf{n/d}-1$. For an integer $k> 2m$ and $v \in \F_{q^n}^*$ we consider 
$$
M_k(\alpha, d;v) =\ssumPI_{\substack{
g_1(\alpha)\ldots g_{k-2m}(\alpha)h_1(\alpha) \ldots h_{2m}(\alpha)= v}}
\Lambda(g_1)\ldots \Lambda(g_{k-2m}) . 
$$
Clearly, if for some $k$ we have $M_k(\alpha, d;v) >0$ for 
every $v \in \F_{q^n}^*$ then $D(\alpha,d) \le k$.

We now closely follow the same path as in the proof of~\cite[Theorem~15]{LWWZ}. 
In particular, using the orthogonality of characters we write
\begin{equation*}
\begin{split}
M_k(\alpha, d;v) &=\frac{1}{q^n-1} \ssumPI 
\Lambda(g_1)\ldots \Lambda(g_{k-2m})\\ & \qquad \qquad \sum_{\chi\in \cX_n}
\chi\(g_1(\alpha)\ldots g_{k-2m}(\alpha)h_1(\alpha) \ldots h_{2m}(\alpha) v^{-1}\).
\end{split}
\end{equation*}

Changing the order of summation, separating the term corresponding to $\chi_0$,
and recalling~\eqref{eq:sum Lambda}, we derive
\begin{equation*}
\begin{split}
 M_k(\alpha, d;v)  - \frac{q^{d(k-2m)}(\# \cI_d)^{2m}} {q^n-1}&\\
 = \frac{1}{q^n-1}&
\sum_{\chi\in \cX_n^*} \chi(v^{-1})
S_{\alpha,d}(\chi)^{k-2m} T_{\alpha,d}(\chi)^{2m}.
\end{split}
\end{equation*}
Therefore
\begin{equation*}
\begin{split}
\left| M_k(\alpha, d;v) - \frac{q^{d(k-2m)}(\# \cI_d)^{2m}} {q^n-1}\right|&\\
\le   \frac{1}{q^n-1}
\sum_{\chi\in \cX_n^*} &
\left| S_{\alpha,d}(\chi)\right|^{k-2m} \left| T_{\alpha,d}(\chi)\right|^{2m}.
\end{split}
\end{equation*}
Using Lemma~\ref{lem:Weil} and then (after extending the summation 
over all $\chi \in \cX_n$) using Lemma~\ref{lem:Moment}, we derive 
\begin{equation}
\label{eq:Mk}
\begin{split}
\left| M_k(\alpha, d;v) - \frac{q^{d(k-2m)}(\# \cI_d)^{2m}} {q^n-1}\right|&\\
\le m! (n-1)^{k-2m} &q^{d(k/2 - m)} (\# \cI_d)^{m}.
\end{split}
\end{equation}
Thus, if for some $v \in \F_{q^n}^*$ we have  $M_k(\alpha, d;v)=0$ 
then
$$
\frac{q^{d(k-2m)}(\# \cI_d)^{2m}} {q^n-1} 
\le  m! (n-1)^{k-2m} q^{d(k/2 - m)} (\# \cI_d)^{m}
$$
or 
\begin{equation}
\label{eq:Prelim}
\(\frac{q^{d/2}}{n-1}\)^k 
\le  m! (n-1)^{-2m} (q^n-1) q^{m} (\# \cI_d)^{-m}.
\end{equation}
Now, as in the proof of~\cite[Theorem~9]{LWWZ}
we note that 
$$
\# \cI_d \ge \frac{q^d}{d} -  \frac{2q^{d/2}}{d}.
$$
Hence~\eqref{eq:Prelim} implies that 
$$
\(\frac{q^{d/2}}{n-1}\)^k 
\le  m! (n-1)^{-2m} d^m (q^n-1) \(1 -  2q^{-d/2}\)^{-m}.
$$
Note that since $n > 2d+1$, we have $m \ge 2$.
Hence,  by the Stirling inequality,  
\begin{equation}
\label{eq:Stirl}
m!\le \sqrt{2\pi} m^{m+1/2} e^{-m+1/12m} \le  \sqrt{2\pi} m^{m+1/2} e^{-m+1/24}.
\end{equation}
Thus, using that $m \le (n-1)/d$, we see that
\begin{equation}
\label{eq:md}
m!d^m  \le  \sqrt{2\pi} m^{1/2} (n-1)^{m}  e^{-m+1/24}.
\end{equation}
Since $d \ge 2$ and $2d+1 \le n <   q^{d/2} + 1$ we have 
$q^{d/2} > 4$. Thus $q^{d/2} \ge 5$.
Furthermore, since $m \le (n-1)/2 < q^{d/2}/2$, we also have 
\begin{equation}
\label{eq:qdm}
 \(1 -  2q^{-d/2}\)^{-m} \le  \(1 -  2q^{-d/2}\)^{-q^{d/2}/2}
 \le  \(1 -  2/5\)^{-5/2}< 3.6.
\end{equation}
Hence, recalling that $m \le (n-1)/d \le (n-1)/2$,  we derive
from~\eqref{eq:md} and \eqref{eq:qdm} that
\begin{equation*}
\begin{split}
\(\frac{q^{d/2}}{n-1}\)^k 
& <   3.6 \sqrt{2\pi} m^{1/2} (n-1)^{-m}   q^n e^{-m+1/24} \\
&\le \sqrt{\pi} (n-1)^{-m+1/2}   q^n e^{-m+1/24}\\ 
&\le \sqrt{\pi} \(e(n-1)\)^{-m+1/2}   q^n e^{-11/24}.
\end{split}
\end{equation*}
Since $m \ge (n-1)/d -1$, we conclude  that 
$$
m-\frac{1}{2} \ge \frac{n}{d} - 2.
$$
Therefore,
$$
\(e(n-1)\)^{-m+1/2} \le \(e(n-1)\)^{-n/d+2},
$$
which finally implies
\begin{equation*}
\begin{split}
k & \le  2\frac{n\log q - (n/d-2)(1+ \log (n-1)) +   \log(3.6\sqrt{\pi}) -11/24}{d \log q - 2\log (n-1)}\\
 & \le  2\frac{n\log q - (n/d-2)(1+ \log (n-1)) + 1.4}{d \log q - 2\log (n-1)}\\
& = \frac{2n}{d}\(1 + \frac{\log(n-1)-1}{d \log q - 2\log(n-1)}\)
+ \frac{4\log(n-1) + 6.8}{d \log q - 2\log(n-1)},
\end{split}
\end{equation*}
which concludes the proof.

\section{Proof of Theorem~\ref{thm:Diam 1}}

We now put   $m = n-1$. Note that the set $\cP_1$ is the set of $q$ linear
polynomials $X+u$, $u \in \F_q$. For an integer $k> 2m$ and $v \in \F_{q^n}^*$ we consider 
$$
N_k(\alpha; v) = \sum_ {\substack{u_1, \ldots,u_k \in \F_q\\
(u_1+\alpha) \ldots (u_{k}+\alpha)= v}}
1 . 
$$
Clearly, if for some $k$ we have $N_k(\alpha;v) >0$ for 
every $v\in \F_{q^n}^*$ then $\Delta(\alpha) \le k$.

Using the same argument as in the proof  Theorem~\ref{thm:Diam d},
we obtain the following analogue of~\eqref{eq:Mk}
$$
\left| N_k(\alpha; v) - \frac{q^{k}} {q^n-1}\right|
\le m! (n-1)^{k-2m} q^{k/2} = 
 (n-1)! (n-1)^{k-2n+2} q^{k/2}. 
$$
Thus if for some $v \in \F_{q^n}^*$ we have  $N_k(\alpha; v)=0$ 
then
\begin{equation}
\label{eq:Prelim2}
\(\frac{q^{1/2}}{n-1}\)^k 
\le  (n-1)! (n-1)^{-2n+2} (q^n-1) .
\end{equation}
The inequality~\eqref{eq:Prelim2} together with the Stirling 
inequality~\eqref{eq:Stirl} imply that, for $n \ge 3$, 
$$
\(\frac{q^{d/2}}{n-1}\)^k 
\le \sqrt{2\pi}  (n-1)^{-n+3/2} q^n  e^{-n+1+1/12(n-1)}.
$$
Using the inequality 
$$
\log\( \sqrt{2\pi}   e^{1+1/12(n-1)}\)  = \frac{25}{24} + 
\frac{1}{2}\log\( 2\pi \) \le 2,
$$
that holds  for $n \ge 3$, we obtain
\begin{equation*}
\begin{split}
k & \le  2\frac{n\log q - (n-3/2) \log (n-1) -n + 2}{\log q - 2\log (n-1)}\\
& = 2n \(1 + \frac{\log(n-1)-1}{\log q - 2\log(n-1)}\)
+ \frac{3\log(n-1) + 2}{ \log q - 2\log(n-1)}, 
\end{split}
\end{equation*}
and the result now follows.

\section*{Acknowledgements}

This work was finished during a very enjoyable 
stay of the author at the Max Planck Institute for Mathematics,
Bonn. 
It was also supported in part by ARC grant~DP140100118.

\end{document}